\documentclass[11pt,twoside]{amsart}

\usepackage{amsfonts}
\usepackage{hyperref}
\usepackage{etoolbox}
\usepackage{amsmath}
\usepackage{amssymb}
\usepackage{amsthm}
\usepackage{dsfont}
\usepackage{pifont}
\usepackage{mathrsfs}
\usepackage{graphicx}
\usepackage{enumerate}
\usepackage[all]{xy}
\usepackage{geometry}
\usepackage{amsfonts}
\usepackage{fancyhdr}
\usepackage{tikz}
\usetikzlibrary{arrows}

\usepackage{amsaddr}

\usepackage[backend=biber,style=alphabetic,sorting=nty,maxnames=5]{biblatex}
\renewbibmacro{in:}{%
  \ifentrytype{article}{}{\printtext{\bibstring{in}\intitlepunct}}}
\newcommand{\Aa}{\mathcal{A}}

\newcommand{\Cc}{\mathcal{C}}

\newcommand{\Gg}{\mathcal{G}}

\newcommand{\Ss}{\mathcal{S}}

\newcommand{\C}{\mathbb{C}}

\newcommand{\N}{\mathbb{N}}

\newcommand{\Z}{\mathbb{Z}}

\renewcommand{\ker}{\mathrm{Ker~}}

\newcommand{\Gl}{\mathrm{GL}}

\newcommand{\infl}{\ar@{{(}->}}
\newcommand{\defl}{\ar@{->>}}

\newcommand{\intv}[1]{[\![#1]\!]}
\newcommand{\muls}[1]{ \left\{\kern-0.6em \left\{ #1\right\}\kern-0.6em \right\} }

\newcommand{\epsi}{\varepsilon}

\newcommand{\xdownarrow}[1]{%
  {\left\downarrow\vbox to #1{}\right.\kern-\nulldelimiterspace}
}

\newcommand{\spa}{\vspace*{1ex}}

\newcommand{\bbar}[1]{\overline{#1}}
\newcommand{\hhat}[1]{\widehat{#1}}
\newcommand{\ttilde}[1]{\widetilde{#1}}
\newcommand{\nit}[1]{\textbf{\emph{#1}}}

\theoremstyle{plain}
\newtheorem{prop}{Proposition}[section]
\newtheorem{prop-def}[prop]{Proposition-Definition}

\newtheorem{lem}[prop]{Lemma}
\newtheorem{theo}[prop]{Theorem}
\newtheorem{cor}[prop]{Corollary}
\newtheorem{rem}[prop]{Remark}
\newtheorem{definition}[prop]{Definition}

 \normalbaselines
\makeatletter
\@namedef{subjclassname@2020}{\textup{2020} Mathematics Subject Classification}
\makeatother

\newcommand{\wid}{\mathrm{wd}}

\title{Regular theory in complex braid groups}
\author{Owen Garnier}
\address{LAMFA, Université de Picardie Jules Verne, CNRS UMR 7352,\\ 33, rue Saint-Leu, 80000, Amiens, France.}
\email{o.garnier@u-picardie.fr}
\date{\today}

\begin{document}

\begin{abstract}
In his seminal paper \cite{beskpi1}, Bessis introduces a Garside structure for the braid group of a well-generated irreducible complex reflection group. Using this Garside structure, he establishes a strong connection between regular elements in the reflection group, and roots of the ``full twist'' element of the pure braid group.

He then suggests that it would be possible to extend the conclusion of this theorem to centralizers of regular elements in well-generated groups.
In this paper we give a positive answer to this question and we show moreover that these results hold for an arbitrary reflection group. As a byproduct, we get a generalization of a theorem from Shvartsman regarding the torsion of the quotient of an irreducible braid group by its center.
\end{abstract}
\maketitle

\tableofcontents

\section{Introduction}

\subsection{Reflection groups, braid groups} We refer to \cite{lehrertaylor} for classical results on complex reflection groups. Let $V$ be a finite-dimensional complex vector space. An element $s\in \Gl(V)$ is called a \nit{(pseudo-)reflection} if $\ker(s-\mathrm{Id}_V)$ is a hyperplane of $V$. We call $\ker(s-\mathrm{Id}_V)$ the \nit{reflecting hyperplane} of $s$. A subgroup $W\leqslant \Gl(V)$ is called a \nit{complex reflection group} if it is finite and generated by reflections of $V$. We say that $n=\dim(V)$ is the \nit{rank} of $W$.

We say that an irreducible complex reflection group $W\leqslant \Gl(V)$ of rank $n$ is \nit{well-generated} if it can be generated by a set of $n$ reflections. Otherwise, we say that $W$ is \nit{badly-generated}.

To a complex reflection group is associated a hyperplane arrangement $\Aa$, defined as the set of reflecting hyperplanes of the reflections of $W$. It is a well-known fact that $W$ acts freely on $X=X(W):=V\setminus \bigcup \Aa$, where $\bigcup \Aa$ denotes the union of the hyperplanes belonging to $\Aa$. We can then define $P(W)=\pi_1(X)$ the \nit{pure braid group} of $W$, and $B(W)=\pi_1(X/W)$ the \nit{braid group} of $W$. The action of $W$ on $X$ gives a covering map $X\twoheadrightarrow X/W$, which in turn induces a short exact sequence.
\[1\to P(W)\to B(W)\to W\to 1\]
We will denote by $\pi$ the projection $B(W)\to W$. To simplify the notation, the orbit of some element $x\in V$ will be denoted by $\bbar{x}\in V/W$.

The definitions we use imply that a complex reflection group is always seen as acting on a finite-dimensional complex vector space. In particular an isomorphism between two reflection groups $W\leqslant \Gl(V)$ and $W'\leqslant \Gl(V')$ is an isomorphism of vector space between $V$ and $V'$, which in turn should induce an isomorphism of groups between $W$ and $W'$. Finite irreducible complex reflection groups were classified (up to isomorphism) by Shephard and Todd:
\begin{enumerate}[\quad -]
\item A general infinite family $G(de,e,n)$, defined for positive integral parameters $d,e,n$, that we call the infinite series.
\item 34 exceptional groups $G_4,\ldots, G_{37}$.
\end{enumerate}
These groups may be separated in two different classes, depending on wether or not they are well generated. In his article \cite{beskpi1}, Bessis extensively studied the case of braid groups of well-generated irreducible reflection groups.
\subsection{Regular elements}
The theory of regular elements of complex reflection groups has been introduced in \cite{sprireg}. Let $d$ be a positive integer. We denote by $\mu_d$ the set of $d$-th roots of unity in $\C$, and $\zeta_d:=e^{\frac{2i\pi}{d}}$.

\begin{definition}
Let $W\leqslant \Gl(V)$ be a complex reflection group and $d$ a positive integer. An element $g\in W$ is called \nit{$\zeta_d$-regular} if it admits regular eigenvectors for the eigenvalue $\zeta_d$, that is if $\ker(g-\zeta_d)\cap X\neq \varnothing$.
\end{definition}

One of the main results of this theory is that $\zeta_d$-regular elements, should they exist, are all conjugate in $W$. Furthermore, for $g$ a $\zeta_d$-regular element, the centralizer $G:=C_W(g)$ of $g$ in $W$ is again a complex reflection group acting on the eigenspace $\ker(g-\zeta_d)$. Let $W'\leqslant \Gl(V')$ be a complex reflection group. We say that $W'$ admits a \nit{regular embedding} in a complex reflection group $W$ if there is some $\zeta_d$-regular element $g$ in $W$ such that $W'$ acting on $V'$ is isomorphic to $C_W(g)$ acting on $\ker(g-\zeta_d)$.

This allows the study of some badly-generated complex reflection groups by means of finding a regular embedding in a well-generated reflection group. A good example of this technique is the embedding of the exceptional group $G_{31}$ in the complexified Coxeter group $G_{37}$, as the centralizer of a $4$-regular element (cf \cite[Example 1.10]{beskpi1}).

In their article \cite[Section 3.B]{brmi}, Broué and Michel predicted the existence of an analogue of Springer regular elements in braid groups: such elements would in turn allow the study of braid groups of some badly-generated reflection groups. A major step towards finding such an analogue is achieved in \cite[Theorem 12.4]{beskpi1}, for the case of well-generated reflection groups. This paper extends the conclusions of Bessis to all complex reflection groups.

In  \cite[Notation 2.3]{bmr}, Broué-Malle-Rouquier consider a particular central element in the pure braid group $P(W)$, represented by the loop $t\mapsto e^{2i\pi t}x_0$ (where $x_0 \in X$ is a basepoint for the pure braid group). It was later shown that, as they conjectured, the center of $P$ is cyclic and generated by this element, that we will denote $z_P$ from now on. Our main theorem is

\begin{theo}Let $W\leqslant \Gl(V)$ be a complex reflection group, and $d$ be a positive integer.
\begin{enumerate}[\quad (A)]
\item There exist $d$-th roots of $z_P$ if and only if $d$ is regular.
\item When $d$ is regular, the $d$-th roots of $z_P$ form a conjugacy class of $B(W)$, and they are mapped to $\zeta_d$-regular elements in $W$.
\item Let $d$ be a regular number. Let also $\rho$ be a $d$-th root of $z_P$, and $w$ its image in $W$. The centralizer $C_{B(W)}(\rho)$ is isomorphic to the braid group $B(W')$ of the centralizer $W':=C_W(w)$.
\end{enumerate}
\end{theo}

A more precise version of property $(C)$ will be given in Proposition \ref{propc} below. In the following we will show that properties $(A),(B)$ and $(C)$ hold for different classes of groups. 

We already know from \cite[Theorem 12.4]{beskpi1} that $(A),(B)$ and $(C)$ hold for a well-generated irreducible reflection group. So it only remains to show that $(A),(B),(C)$ hold for badly-generated reflection groups. Our proof consists of the following steps:

After recalling known useful results in Sections 2 and 3, we prove in Section 4 that properties $(A),(B),(C)$ hold for the members of the infinite family that are badly-generated (that is the $G(de,e,n)$ where $d\geqslant 2$ and $e\geqslant 2$). We then show in Section $5$ that, if $W$ is a reflection group for which $(A),(B),(C)$ hold, then $(A),(B),(C)$ also hold for every centralizer of a regular element of $W$. In Section 6, we show that if $W$ and $W'$ are reflection groups with the same  (reflection) degrees and (reflection) codegrees, then properties $(A),(B),(C)$ for $W$ imply properties $(A),(B),(C)$ for $W'$. Lastly, the two remaining groups, $G_{12}$ and $G_{13}$, are studied separately in Section $7$.

This work will be a part of my PhD thesis, done under the supervision of Ivan Marin. I thank him very much for his precious help, especially in Section \ref{serieinf}.

\section{Preliminary results}
From now on, $W$ will denote an irreducible complex reflection group, $\Aa$ its hyperplane arrangement, and $d$ a positive integer. We also choose a basepoint $x_0$ in $X$ and we set 
\[P(W):=\pi_1(X,x_0),~~B(W):=\pi_1(X/W,\bbar{x_0})\]
If $x_1$ is another basepoint, then since $X$ is path-connected, we can always consider a path from $x_0$ to $x_1$. Such a path induces an isomorphism of short exact sequences
\[\xymatrix{1\ar[r]&\ar[r] P(W)\ar[d]^\simeq \ar[r] & B(W)\ar[r] \ar[d]^\simeq &W \ar[r] \ar@{=}[d] & 1\\ 1\ar[r] &\pi_1(X,x_1)\ar[r] & \pi_1(X/W,\bbar{x_1}) \ar[r] & W\ar[r] & 1}\]
which shows that the change of basepoint in $X$ yields an isomorphism of groups over $W$ between $\pi_1(X/W,\bbar{x_1})$ and $B(W)=\pi_1(X/W,\bbar{x_0})$.
\subsection{Centers of complex braid groups}
A long-standing task in braid group theory has been to determine the centers of both $B(W)$ and $P(W)$. A final statement has been obtained in \cite{dmm}, completing previous results and conjectures (see for instance \cite[Theorem 2.24]{bmr} and \cite[Theorem 12.8]{beskpi1}):

\begin{theo}(\cite[Theorems 1.1, 1.2, 1.3]{dmm})\newline Assume that $W$ is irreducible. The centers of $P(W),B(W)$ and $W$ are cyclic, respectively generated by elements $z_P,z_B,\pi(z_B)$ satisfying $z_B^{|Z(W)|}=z_P$. In particular, we have a short exact sequence
\[1\to Z(P(W))\to Z(B(W))\to Z(W)\to 1.\]
\end{theo}
The element $z_P$ is called the \nit{full-twist} of $W$, and it has a topological origin: it is the homotopy class of the loop $z_{x_0}:t\mapsto e^{2i\pi t}x_0$. If $x_1$ is another basepoint, and $\gamma$ is a path from $x_0$ to $x_1$, then it is easy to show that the concatenated paths $\gamma * z_{x_1}$ and $z_{x_0} * \gamma$ are homotopic. This means that the isomorphism between $B(W)=\pi_1(X/W,\bbar{x_0})$ and $\pi_1(X/W,\bbar{x_1})$ sends $z_P$ to the homotopy class of $z_{x_1}$. This allows us to define $z_P$ without mentioning the basepoint.

This full-twist will play a key role in defining analogues to regular elements, which will be the roots of $z_P$. The first important task is to define, for regular elements, lifts which are roots of $z_P$. We recall this construction now.
 
Let $g\in W$ be a $\zeta_d$-regular element and $x\in X$ be a $\zeta_d$-regular eigenvector for $g$. We define a path in $X$ by
\[\ttilde{g}:t\mapsto e^{\frac{2i\pi t}{d}}x\]
where $t$ ranges between $0$ and $1$. Since this path ends at $\zeta_d x=g.x$, it induces a well-defined element of $\pi_1(X/W,\bbar{x})$. This element is a $d$-th root of $z_P$ and a lift of $g$.
Now if $y\in X$ is another basepoint, then we choose a path from $x$ to $y$ in $X$. We get an isomorphism $\pi_1(X/W,\bbar{x})\simeq \pi_1(X/W,\bbar{y})$, which sends the homotopy class $[\ttilde{g}]$ of $\ttilde{g}$ to an element of $\pi_1(X/W,\bbar{y})$, which is also a $d$-th root of $z_P$ and a lift of $g$. 

 In general, a lift of $g$ in $\pi_1(X/W,\bbar{y})$ which is a $d$-th root of $z_P$ will be called a \nit{regular lift} of $g$ (in $\pi_1(X/W,\bbar{y})$).

\subsection{A precise rephrasing of property $(C)$}
The isomorphism whose existence is asserted by property $(C)$ is explicit, and can be constructed as follows:

Let $g$ be a $\zeta_d$-regular element in $W$. Let $V_g=\ker(g-\zeta_d)$ be the space on which $G:=C_W(g)$ acts as a complex reflection group. The hyperplane arrangement for $G$ acting on $V_g$ is $\Aa_g:=\{H\cap V_g~|~ H\in \Aa\}$ where $\Aa$ is the hyperplane arrangement for $W$. We deduce that
\[X_g:=V_g\setminus\bigcup \Aa_g=V_g\setminus\left(\left(\bigcup \Aa\right) \cap V_g\right)=V_g\cap X.\]
The embedding $X_g\to X$ induces a homeomorphism $p:X_g/G\to (X/W)^{\mu_d}$: Denoting $\hhat{x}=G.x$ and $\bbar{x}:=W.x$, we have $p(\hhat{x})=\bbar{x}$. The inverse bijection of $p$ is a little bit harder to compute. Let $\bbar{x}\in (X/W)^{\mu_d}$, this means that $\zeta_d\bbar{x}=\bbar{\zeta_d x}=\bbar{x}$, so there is some $g'\in W$ such that $g'(x)=\zeta_d x$. Since $x\in X$, we get that $g'$ is also a $\zeta_d$-regular element of $W$. We obtain that $w^{-1}gw=g'$ for some $w\in W$, so 
\[w^{-1}gw.x=\zeta_d x\Rightarrow g.(w.x)=\zeta_d(w.x)\]
this proves $w.x\in X_g$, and so $p^{-1}(\bbar{x})=\hhat{w.x}$. The following lemma is well-known (see for instance \cite[Theorem-Assumption 2.25]{broue}). We give a detailed proof here for the sake of clarity.

\spa\begin{lem}\label{2.2}
The map $X_g/G\to (X/W)^{\mu_d}\hookrightarrow X/W$ induces a group morphism
\[B(G)=\pi_1(X_g/G,\hhat{x})\to \pi_1(X/W,\bbar{x})\]
whose image lies inside $C_{\pi_1(X/W,\bbar{x})}([\ttilde{g}])$, where $[\ttilde{g}]$ is the regular lift of $g$ in $\pi_1(X/W,\bbar{x})$ constructed above.
\end{lem}
\begin{proof}
Let $x\in X_g$ be our basepoint. Let us recall that, since we have a covering map $X_g\twoheadrightarrow X_g/G$, an loop $\gamma$ from $\bbar{x}$ to $\bbar{x}$ in $X_g/G$ induces a unique  well-defined loop $\ttilde{\gamma}$ in $X_g$, starting from $x$, and ending at some $w.x$, where $w\in W$. Furthermore the homotopy class $[\ttilde{\gamma}]$ depends only on $[\gamma]\in \pi_1(X_g/G,\bbar{x})$. The path $w.\ttilde{\gamma}$ is then the unique lift of $\gamma$ starting at $w.x$. 

We know that $g$ acts on $V_g$ by multiplication by $\zeta_d$, the path $\ttilde{g}$ defined above is a path inside $X_g$, and it induces an element $[\ttilde{g}]$ in $\pi_1(X_g/G,\hhat{x})$ (which is by construction a regular lift of $g$). Let now $[\gamma]$ be another element of $\pi_1(X_g/G,\hhat{x})$,  the loop $\gamma$ lifts to a unique path $\ttilde{\gamma}$ in $X_g$ from $x$ to some $w.x$ (with $w\in G$). The product $[\gamma][\ttilde{g}]$ is defined using the lift 
\[\ttilde{\gamma}*(w.\ttilde{g}):t\mapsto \begin{cases}  \ttilde{\gamma}(2t)&\text{if }t\leqslant 1/2\\ w.\ttilde{g}(2t-1)=e^{\frac{2i\pi (2t-1)}{d}}w.x&\text{if }t\geqslant 1/2\\\end{cases}\]
And the product $[\ttilde{g}][\gamma]$ is defined using the lift
\[\ttilde{g}*(g.\ttilde{\gamma}):t\mapsto \begin{cases} \ttilde{g}(2t)=e^{\frac{4i\pi}{d}}.x&\text{if }t\leqslant 1/2\\ g.\ttilde{\gamma}(2t-1)=\zeta_d.\ttilde{\gamma}(2t-1)&\text{if }t\geqslant \frac{1}{2}\end{cases}\]
Both of these paths are homotopic to 
\[f:t\mapsto e^\frac{2i\pi t}{d} \ttilde{\gamma}(t)\]
therefore $[\gamma][\ttilde{g}]=[\ttilde{g}][\gamma]$ in $\pi_1(X_g/G,\hhat{x})$.
\end{proof}

Composing the morphism $B(G)\to C_{\pi_1(X/W,\bbar{x})}([\ttilde{g}])$ with an isomorphism $\pi_1(X/W,\bbar{x})\to B(W)$ coming from a path from $x$ to $x_0$ in $X$ yields a morphism $B(G)\to C_{B(W)}([\ttilde{g}])$, where $[\ttilde{g}]$ is a regular lift of $g$ in $B(W)$. This is the morphism we need in order to properly state property $(C)$.

\spa\begin{prop}\label{propc} Let $W$ be a complex reflection group for which properties $(A)$ and $(B)$ hold. Let also $\rho$ be a $d$-th root of $z_P$, let $g=\pi(\rho)$ be its image in $W$, and let $G:=C_W(g)$. Then the morphism $B(G)\to C_{B(W)}([\ttilde{g}])$ defined above is an isomorphism. Since $[\ttilde{g}]$ and $\rho$ are conjugate in $B(W)$ by property $(B)$, we deduce an isomorphism of groups over $G$ between $B(G)$ and $C_{B(W)}(\rho)$.
\end{prop}

This proposition is what we will call point $(C)$ from now on, we will prove it for all badly-generated reflection groups in the subsequent sections.

\subsection{Garside theory}
As we said, a key ingredient in the proof of \cite[Theorem 12.4]{beskpi1} is the dual braid monoid, which is a Garside monoid for the braid group of a well-generated complex reflection group. Our approach will also use Garside structures and tools from Garside theory. We give a quick introduction here, in order to settle notation and to state the main results we are going to use. Our main reference is \cite{ddgkm}.
\begin{definition}(\cite[Definition I.2.1]{ddgkm})\newline 
Let $M$ be a monoid. Assume that
\begin{enumerate}[\quad -]
\item $M$ is \nit{homogeneous}, that is, there is some monoid morphism $\ell:M\to \N$ such that $M$ is generated by elements of positive length.
\item $M$ is left and right cancellative
\item the posets $(M,\prec)$ and $(M,\succ)$ of left and right divisibility are both lattices (that is every pair of elements admits left and riqht lcms and gcds)
\end{enumerate}
On such a monoid, a \nit{Garside structure} is given by an element $\Delta$, called a \nit{Garside element}, such that the sets of left and right divisors of $\Delta$ are equal, finite, and generate $M$. This set of divisors will be noted $\Ss$, and its elements will be called \nit{simples} of the Garside structure.
\end{definition}
This definition is far from minimal, but it is sufficient for the study of most braid groups. One can consult \cite{ddgkm} for a more general setup. 

We recall three important properties of Garside monoids:
\begin{enumerate}[\quad -]
\item A Garside monoid $(M,\Delta)$ embeds in its envelopping group $\Gg(M)$ (that can be accurately described as a group of fractions of $M$). We say that $\Gg(M)$, endowed with $M$ and $\Delta$, is a \nit{Garside group}.
\item Conjugation by $\Delta$ in $\Gg(M)$ restricts to an automorphism $\phi$ of $M$, defined by $\Delta\phi(x)=x\Delta$ for $x\in M$. We call this automorphism the \nit{Garside automorphism} of the Garside monoid $(M,\Delta)$.
\item If $(M,\Delta)$ is a Garside monoid, then $(M,\Delta^k)$ is also a Garside monoid for $k\in \Z_{>0}$.
\end{enumerate}

The feature of Garside structures on braid groups that we are going to use is that $z_P$ is equal to some power of $\Delta$. That way a root of $z_P$ becomes a \nit{periodic element} in a Garside structure, in the sense of \cite[Definition VIII.3.2]{ddgkm}. An element $\gamma$ in a Garside group is $(p,q)$-periodic if
\[\gamma^p=\Delta^q\]
In general we will say that an element is periodic if it is $(p,q)$-periodic for some couple $(p,q)$. It is obvious that a $(p,q)$-periodic element is also $(np,nq)$-periodic for every integer $n$. There is some kind of converse statement, which will be useful for reducing computations:
\begin{prop}\label{334}(\cite[Proposition VIII.3.34]{ddgkm})\newline
Let $B$ be a Garside group, and let $\gamma\in B$ be a $(p,q)$-periodic element. We denote by $p\wedge q$ the gcd of $p$ and $q$.
\begin{enumerate}[\quad (a)]
\item Up to conjugacy, we can assume that $\gamma^{p'}=\Delta^{q'}$, where $p'=\frac{p}{p\wedge q}$ and $q'=\frac{q}{p\wedge q}$ and for all $(u,v)\in \Z_{\geqslant 0}\times \Z_{\leqslant 0}$ such that $p'u+q'v=1$, we have $\gamma^{v}\Delta^u\in \Ss$.
\item Furthermore, the value of $g=\gamma^v\Delta^u$ does not depend on the choice of the pair $(u,v)$ and we have $\gamma^{p'v}\Delta^{p'u}=\Delta$, $\gamma^{q'v}\Delta^{q'v}=\gamma$.
\item In particular, if $\Delta$ and $\gamma$ commute, we have $g^{p'}=\Delta$ and $g^{q'}=\gamma$.
\end{enumerate}  
\end{prop}
The last tool we need to introduce is the \nit{divided category} (see \cite[Definition XIV.1.2]{ddgkm}), which will allow us to explicitly compute presentations of centralizers of regular elements when needed. Let $(M,\Delta)$ be a Garside monoid with set of simples $\Ss$. If $m,n$ are two integers, we set
\[D_m^n(\Delta):=\left\{ (a_1,\ldots,a_m)\in \Ss^m~|~ \prod_{i=1}^m a_i=\Delta~\text{and~} (a_1,\ldots,a_m)^{\phi^n}=(a_1,\ldots,a_m)\right\}\]
where $\phi$ acts on tuples by 
\[(a_1,\ldots,a_m)^\phi:=(a_{2},\ldots,a_m,a_1^\phi)\]
For instance, we have that $D_2^0(\Delta)$ is in bijection with the set $\Ss$, or that $D_2^1(\Delta)$ is in bijection with elements $s$ such that $s^2=\Delta$. 

For a couple of integers $(p,q)$, we define a category $\Cc_p^q$ by using the sets $D_p^q(\Delta), D_{2p}^{2q}(\Delta)$ and $D_{3p}^{3q}(\Delta)$, considered respectively as a set of objects, generating morphisms, and relations:
\begin{enumerate}[\quad -]
\item An element $(a_{1},\ldots,a_{2p})\in D_{2p}^{2q}(\Delta)$ seen as a morphism has source and target as follows
\[(a_1,\ldots,a_{2p}): (a_1a_2,\ldots,a_{2p-1}a_{2p})\to(a_2a_3,\ldots,a_{2(p-1)}a_{2p-1},a_{2p}a_1^\phi),\]
one can easily check that both the source and the target are elements of $D_p^q(\Delta)$.
\item An element $(a_{1},\ldots,a_{3p})\in D_{3p}^{3q}(\Delta)$ induces the following relation
\[(a_1,a_2a_3,\ldots,a_{3p-2},a_{3p-1}a_{3p})(a_2,a_3a_4,\ldots,a_{3p-1},a_{3p}a_1^\phi)=(a_1a_2,a_3,\ldots,a_{3p-2}a_{3p-1},a_{3p})\]
between three elements of $D_{2p}^{2q}(\Delta)$.
\end{enumerate}

\spa\begin{theo}(\cite[Proposition XIV.1.8]{ddgkm})\label{divi}
\begin{enumerate}[\quad (a)]
\item The map $(a_1,\ldots,a_{2p})\mapsto a_1$ induces a functor from $\Cc_p^q$ to $M$ (where $M$ is seen as a category with one object). It will be called the \nit{collapse functor} (following \cite[Definition B.23]{beskpi1}).
\item The (undirected) connected components of $\Cc_p^q$ (seen as an undirected graph) are in one to one correspondance with conjugacy classes of $(p,q)$-periodic elements.
\item Consider the enveloping groupoid $\Gg_p^q$ of $\Cc_p^q$, and $x$ an object in this groupoid. The collapse functor induces a functor $\Gg_p^q\to \Gg(M)$, which sends the group $\Gg_p^q(x,x)$ of morphisms from $x$ to itself in $\Gg_p^q$ to the centralizer of some $(p,q)$-periodic element of $\Gg(M)$.
\end{enumerate}
\end{theo}
In order to prove points $(B)$ and $(C)$ of our main theorem, it is enough to compute directly a presentation of the category $\Cc_p^q$. We can then see that this category is connected, which will prove point $(B)$, and use its presentation to deduce a presentation of the centralizer of some $(p,q)$-periodic element in order to prove point $(C)$. Here we see a first use of Proposition \ref{334}: replacing $(p,q)$ by $\left(\frac{p}{p\wedge q},\frac{q}{p\wedge q}\right)$ makes the set $D_{kp}^{kq}(\Delta)$ a lot easier to compute in practice.

\section{First reductions for the proof}
We begin our proof by noticing that some implications in the theorem are obvious:
\begin{enumerate}[\quad -]
\item If $d$ is a regular number for $W$, then there exists $d$-th roots of $z_P$. Indeed we have seen that for instance, a regular lift $\ttilde{g}$ of a regular element $g\in W$ is such a root.
\item If the $d$-th roots of $z_P$ are conjugate, and if $d$ is regular, then all $d$-th roots of $z_P$ are conjugate to every regular lift of a $\zeta_d$-regular element. So they are mapped to $\zeta_d$-regular elements of $W$.
\end{enumerate}
Thanks to this, in practice we will only need to show that, if $d$-th roots of $z_P$ exist, then $d$ is regular, that $d$-th roots of $z_P$ are conjugate, and that the morphism of Lemma \ref{2.2} is an isomorphism.

There are 2 cases (namely those of $G_{12}$ and $G_{13}$) that we will solve through explicit computations. In order to reduce these computations to a minimum, we give here some results on regular elements and regular numbers. These results hold in a general setting, but we will mostly use them on groups of rank 2.

It is known that the existence of regular elements for a complex reflection group is conditioned by the degrees and codegrees of $W$ (see for instance \cite[Theorem 11.28]{lehrertaylor}). When $W$ is irreducible,  we define $A(d)$ (resp. $B(d)$) as the degrees (resp. codegrees) of $W$ that are divisible by $d$. The number $d$ is regular for $W$ if and only if $|A(d)|=|B(d)|$. When it is the case, the centralizer of a $\zeta_d$-regular element is a reflection group, whose degrees (resp. codegrees) are the elements of $A(d)$ (resp. of $B(d)$).

This characterization in terms of divisibility has a first obvious consequence. If $d$ is a regular number for $W$, then we have $|A(d)|=|B(d)|$. Moreover, if we define $d'$ as the gcd of the union $A(d)\cup B(d)$, then we get that $d|d'$ (by definition of a gcd), and that $d'$ is regular for $W$, with $A(d')=A(d)$ and $B(d')=B(d)$ (and $d'$ is the greatest integer with this property). The number $d'$ will be called the \nit{fundamental regular number} associated to $d$. From this we get a bijection between conjugacy classes of centralizers of regular elements in $W$ and fundamental regular numbers of $W$. 

Consider $d_1<\ldots<d_k$ the fundamental regular numbers of $W$. The regular numbers of $W$ are exactly the integers that divide one of the $d_i$. The fundamental regular number $d'$ associated to $d$ is the smallest $d_i$ such that $d|d_i$. We get
\[\forall d\in \N, i\in \intv{2,k},~(A(d)=A(d_i)\text{~and~}B(d)=B(d_i))\Leftrightarrow (d|d_i\text{~and~}d\nmid d_{i-1})\]
We write $R_i$ for the set of integers $d$ satisfying those conditions (we set for $R_1$ the divisors of $|Z(W)|$, the fundamental regular number associated to $1$).

Consider now $d\in R_i$ some regular number, we could look for a smallest number $d'' \in R_i$ dividing $d$. This exists by definition of $R_i$, but it depends on $d$. Consider for instance the dihedral group $G(12,12,2)$ of order 24. Its degrees and codegrees are $2,12$ and $0,10$. So $3$ and $4$ are coprime regular numbers, both with fundamental regular number $12$, which is not the fundamental regular number for $1$. However, it is easy to check by computer that the $R_i$ are lattices when $W$ is an exceptional group.

\spa\begin{lem}\label{fund}
Let $g\in W$ be a $\zeta_d$-regular element of $W$, $d'$ the fundamental regular number associated to $d$, and let $d''$ be its smallest associated regular number. There exists some $\zeta_{d'}$-regular element $g'$ such that $g'^{d'/d}=g$, and with 
\[C_W(g')=C_W(g)=C_W(g^{\frac{d}{d''}})\]
\end{lem}
\begin{proof}
Let $g$ be a $\zeta_d$-regular element, and $g_d$ be a $\zeta_{d'}$-regular element. We know that $g_d^{d'/d}$ is $\zeta_d$-regular, so there is some $a\in W$ such that $ag_d^{d'/d}a^{-1}=g$. The element $g':=ag_d a^{-1}$ is $\zeta_{d'}$-regular and we have $g'^{d'/d}=g$ by definition. Moreover, the equality $C_W(g')=C_W(g)$ is obvious since we have $C_W(g')\subset C_W(g)$, and they are both reflection groups with the same degrees. The second equality of centralizers is obvious for the same reasons.
\end{proof}

\begin{rem}(Fundamental roots)\label{fr}\newline We assume that properties $(A)$ and $(B)$ hold for $W$. Let $\rho$ be a $d$-th root of $z_P$, and $g=\pi(\rho)$. We have that $g$ is $\zeta_d$-regular, so there is some $\zeta_{d'}$-regular element $g'$ that is a root of $g$ by Lemma \ref{fund}, where $d'$ is the fundamental regular number associated to $d$. A regular lift of $g'$ in $B(W)$ is by construction a $\tfrac{d'}{d}$-th root of some regular lift $\ttilde{g}$ of $g$. Since $\ttilde{g}$ and $\rho$ are conjugate, we get that there is some $\rho'$ such that $\rho'^{d'}=z_P$ and $\rho'^{d'/d}=\rho$. We call $\rho'$ a \nit{fundamental root} of $z_P$, associated to $\rho$.
\end{rem}

\begin{rem} In \cite[Definition 3]{shvartsman} is introduced the notion of \nit{regular Springer set} to denote the set of regular numbers that are maximal with respect to divisibility. Obviously the elements of the regular Springer set are fundamental degrees, although the converse is false in general. The proof of Lemma \ref{fund}  also proves that any $\zeta_d$-regular element is a power of a $\zeta_{d'}$-regular element where $d'$ is in the regular Springer set. The same goes for Remark \ref{fr}: $d$-th root of $z_p$ in the braid group is a power of a $d'$-th root of $z_p$, where $d'$ is in the regular Springer set. The notion of regular Springer set will be used again in Section \ref{sec8}.
\end{rem}

\begin{rem}\label{r2}(Property $(C)$ for groups of rank two)\newline Let $W$ be of rank 2 and such that properties $(A)$ and $(B)$ hold. We denote by $d_1,d_2$ the degrees of $W$ and by $d^*$ its non zero codegree. If $g\in W$ is a $\zeta_d$-regular element, then we have either
\begin{enumerate}[\quad -]
\item $d$ divides $d_1,d_2,d^*$, in which case $C_W(g)=W$, $V_g=V$, $X_g=X$ and property $(C)$ is obvious since $X/W$ and $X_g/G$ are naturally homeomorphic.
\item $d$ divides $d_1$ and doesn't divide $d_2$ and $d^*$. Then $d_1$ doesn't divide $d^*$, and it is a regular number: the fundamental regular number associated to $d$. If $\rho$ is a $d$-th root of $z_P$, and $\rho'$ a fundamental root associated to $\rho$, then we have to show that $C_{B(W)}(\rho)=C_{B(W)}(\rho')=\langle \rho'\rangle$ to prove property $(C)$. Thanks again to Lemma \ref{fund}, we can always assume that $d$ is equal to its own smallest associated number.
\item $d$ divides $d_2$ and doesn't divide $d_1$ and $d^*$. This corresponds to the above case when swapping $d_1$ and $d_2$ (as we did not assume that $d_1<d_2$ here).
\end{enumerate}
\end{rem}

\section{Infinite series}\label{serieinf}
This section consists in showing that properties $(A),(B)$ and $(C)$ hold for $W=G(de,e,n)$ a member of the infinite series. We need not consider the case where $d=1$ or $e=1$, as they give rise to well-generated reflection groups. 

As is often the case, $G(de,e,n)$ will be tackled by tacking advantage of the natural embedding $G(de,e,n)\hookrightarrow G(de,1,n)$ making $G(de,e,n)$ a normal subgroup of $G(de,1,n)$ (which is well-generated).

The first two points of our theorem for the infinite series are already considered in \cite[Discussions after Theorem 4.15]{cll}. We borrow some of their notation and reproduce a synthetic proof for the sake of clarification. 

We recall from \cite[Theorem 3.6]{bmr} that the group $B(de,1,n)\simeq B(2,1,n)$ is generated by $b_1,\ldots,b_n$ with the relations 
\begin{enumerate}[$(B_1)$]
\item $b_1b_2b_1b_2=b_2b_1b_2b_1$
\item $b_ib_{i+1}b_i=b_{i+1}b_ib_{i+1}$ for $i\in \intv{2,n-1}$
\item $b_ib_j=b_jb_i$ for $|i-j|>1$
\end{enumerate}
And that $B(de,e,n)$, seen as a subgroup of $B(de,1,n)$, is generated by
\[\begin{array}{ccc} z:=b_1^e & t_i:=b_1^{-i}b_2b_1^i (i \in \Z) & s_j=b_j (j\geqslant 3)\end{array}\]
We consider the morphism\footnote{$\mathrm{wd}$ stands for ``winding number'', a notation we borrow from \cite[Definition 2.2]{cll}} $\wid:B(de,1,n)\to \Z$ defined by 
\[\wid(b_1)=1\text{~and~}\wid(b_i)=0\text{ for }i\in \intv{2,n}.\]
We have that $b$ is in $B(de,e,n)$ if and only if $\wid(b)\equiv 0[e]$. We will denote $\wid_e$ the morphism $b\mapsto \wid(b)[e]$, so we have $\ker \wid_e=B(de,e,n)$.

We first notice that, as the hyperplane arrangements for $G(de,e,n)$ and $G(de,1,n)$ in $\C^n$ are the same (since $d\geqslant 2$), the element $z_P$ is the same for $B(de,e,n)$ and $B(de,1,n)$. So there is no conflict of notation between $z_{P(de,e,n)}$ and $z_{P(de,1,n)}$.

In $G(de,1,n)$, the element $\epsilon:=\pi(b_1\ldots b_n)$ is equal to
\[\epsilon=\begin{pmatrix}0 & 0 & \cdots & 0 & \zeta_{den}^n\\ 1 &0 & \cdots&0 & 0\\0 & 1 & \cdots & 0 & 0\\ \vdots & \vdots & \ddots & \vdots & \vdots \\ 0 & 0 & \cdots & 1 & 0\end{pmatrix} \]
It admits $(\zeta_{den}^{n-1},\ldots,\zeta_{den})$ as a $\zeta_{den}$ eigenvector: it is a $\zeta_{den}$-regular element in $G(de,1,n)$, that is a generalized Coxeter element in the sense of \cite[Definition 7.1]{beskpi1}. As such, it induces a Garside element $\epsi=b_1\cdots b_n$ defining a dual structure of $B(de,1,n)$ (see \cite[Definition 8.7 and Lemma 8.8]{beskpi1}). We have
\[\varepsilon^{n}=z_{B(de,1,n)}\text{~~and~~}\varepsilon^{den}=z_P\]
On the other hand, we know that the degrees and codegrees of $B(de,1,n)$ are
\[\begin{array}{cccc}  de& 2de&\ldots & den\\ 0 &de &\ldots &de(n-1) \end{array}\]
So a number $k$ is regular for $G(de,1,n)$ if and only if it divides $den$. Applying point $(B)$ to the well-generated group $G(de,1,n)$, we get

\begin{lem}\label{5.1}
Every periodic element in $B(de,1,n)$ is conjugate to some power of $\varepsilon$.
\end{lem}
\begin{proof} Let $\gamma\in B(de,1,n)$ be such that $\gamma^p=\varepsilon^q$. Up to conjugacy, we can assume in particular that $p$ and $q$ are coprime (thanks to Proposition \ref{334} applied to the dual Garside structure for $B(de,1,n)$). Let $k$ be the gcd of $q$ and $den$, we have
\[\gamma^{\frac{pden}{k}}=\varepsilon^{\frac{qden}{k}}=z_P^{\frac{q}{k}}\]
We can then consider $z_P=\varepsilon^{den}$ as a Garside element for another Garside structure on $B(de,1,n)$. Using again Proposition \ref{334}, $\gamma$ is conjugate to some $\ttilde{\gamma}$ such that
\[\ttilde{\gamma}^{\frac{pden}{ka}}=z_P^{\frac{q}{ka}}\]
where $a$ is the gcd of $\frac{q}{k}$ and $\frac{pden}{k}$ (so $ka=q\wedge (pden)$). Another use of Proposition \ref{334}, gives us a $\frac{pden}{ka}$-th root of $z_P$ (since $z_P$ is central). By \cite[Theorem 12.4]{beskpi1} we obtain $\frac{pden}{ka}|den\Rightarrow p|ka$.

 But since $ka$ divides $q$, we get that $p$ divides $q$, and since they are coprime, we have $p=1$. Up to conjugacy, we can assume that a periodic element $\gamma$ is such that $\gamma=\varepsilon^q$, which is what we wanted.
\end{proof}

This generalizes property $(B)$ for $B(de,1,n)$ to all periodic elements instead of just roots of $z_P$, which are a certain kind of periodic elements.

Since $\mathrm{wd}(\varepsilon)=1$, we have $\lambda:=\varepsilon^e\in B(de,e,n)$, this element is a $nd$-th root of $z_P$. The following proposition has been proven in \cite[Theorem 4.14]{cll}, using the tools they introduce there. We provide a more direct proof here:

\begin{prop}
Every element in $B(de,e,n)$ that admits a central power is conjugate in $B(de,e,n)$ to some power of $\lambda$. 
\end{prop}
\begin{proof}
Thanks to \cite[Theorem 1.4]{dmm}, since $B(de,e,n)$ is a finite index subgroup of $B(de,1,n)$, we have $Z(B(de,e,n))\subset Z(B(de,1,n))$. So if $\rho\in B(de,e,n)$ admits a central power in $B(de,e,n)$, then it also does in $B(de,1,n)$. Since the center of $B(de,1,n)$ is generated by $\varepsilon^n$, we deduce that $\rho$ is periodic in $B(de,1,n)$, the last proposition then implies that $\rho$ is conjugate in $B(de,1,n)$ to some $\varepsilon^r$. 

 Since $B(de,e,n)$ is normal, we must have $\varepsilon^r\in B(de,e,n)$, that is $\wid(\varepsilon^r)=r\equiv 0[e]$. This implies that $r=pe$ for some integer $p$ and that there is some $g\in B(de,1,n)$ such that $\rho^g=\varepsilon^{pe}=\lambda^p$ so $\rho$ is conjugate to some power of $\lambda$ in $B(de,1,n)$. We need to show that the conjugating element $g$ need is in $B(de,e,n)$. As $\varepsilon$ is an element of winding number $1$ that centralizes $\lambda$, so assuming $i=\wid(g)$, we get
\[\rho^{g\varepsilon^{-i}}=(\lambda^{\varepsilon^{-i}})^p=\lambda^p\]
and $g\varepsilon^{-i}\in B(de,e,n)$ is an element of $B(de,e,n)$ that conjugates $\rho$ to $\lambda^p$.
\end{proof}

\begin{cor}
The properties $(A)$, $(B)$ and $(C)$ are true for $B(de,e,n)$ when $d,e\geqslant 2$.
\end{cor}
\begin{proof}
Let $\rho\in B(de,e,n)$ be such that $\rho^k=z_P$. By the last proposition $\rho$ is conjugate to some $\lambda^r$ and we have $\lambda^{rk}=z_P=\lambda^{dn}$. Since $B(de,e,n)$ has no torsion, we have $rk=dn$ and $k$ divides $dn$, which gives point $(A)$.

For point $(B)$, any two $k$-th roots of $z_P$ are conjugate to the same power of $\lambda$, and so they are conjugate. As powers of $\lambda$ are mapped in $G(de,e,n)$ to some conjugate of a regular element, it is also the case of any $k$-th root of $z_P$.

For the last point, we only have to show the result for powers of $\lambda$. Set $W=G(de,e,n)$, $\hhat{W}=G(de,1,n)$, $B$ and $\hhat{B}$ their respective braid groups, $G=C_W(\pi(\lambda)^p)$ and $\hhat{G}=C_{\hhat{W}}(\pi(\lambda)^p)$. We have
\[C_B(\lambda^p)=C_{\hhat{B}}(\lambda^p)\cap B=C_{\hhat{B}}(\lambda^p)\cap \ker (\wid_e)=\ker(\wid_{e|C_{\hhat{B}}(\lambda^p)})\]
\[G=\hhat{G}\cap W=\hhat{G}\cap \ker (\bbar{\wid_e})=\ker(\bbar{\wid_{e}}_{|\hhat{G}})\]
We already know that the morphism $B(\hhat{G})\to C_{\hhat{B}}(\lambda^p)$ is an isomorphism. Furthermore, since $\varepsilon$ commutes with $\lambda$ and has winding number $1$, the morphisms $\bbar{\wid_{e}}_{|C_{\hhat{W}}(\pi(\lambda)^p)}$ and $\wid_{e|C_{\hhat{B}}(\lambda^p)}$ are surjective, and we have short exact sequences
\[\xymatrixcolsep{3pc}\xymatrix{1 \ar[r] & B(G)\ar[ddd] \ar[r] & B(\hhat{G})\ar[ddd]^-{\simeq} \ar[rd] \ar[rr]^-{\wid_e} &  & \Z/e\Z\ar@{=}[ddd] \ar[r]& 1\\ & & & \hhat{G}\ar[d] \ar[ru]_-{\bbar{\wid_e}} & & \\ & & & \hhat{W} \ar[rd]^-{\bbar{\wid_e}} & & \\ 1 \ar[r] & C_B(\lambda^p) \ar[r] & C_{\hhat{B}}(\lambda^p)\ar[ru] \ar[rr]_-{\wid_e} & & \Z/e\Z\ar[r]  & 1}\]
We deduce from this that the morphism $B(G)\to C_B(\lambda^p)$ is an isomorphism. This proves property $(C)$ for $B(de,e,n)$.
\end{proof}

\section{Centralizers of regular elements}
We now turn to the case of reflection groups that appear as centralizers of regular elements. More precisely we show that if properties $(A),(B),(C)$ hold for a given group $W$, then they also hold for centralizers of regular elements in $W$ (this idea originates from \cite[Remark 12.6]{beskpi1}).

Let $W$ be a reflection group for which properties $(A),(B)$ and $(C)$ are true, and let $\delta$ be a $r$-root of $z_P$ in $B(W)$. We set $G=C_W(\pi(\delta))$, we know by hypothesis that $B(G)$ identifies to $C_{B(W)}(\delta)$, identification which sends $z_P$ on the full twist defined for $W$. Let $d$ be another integer, and set $k=d\wedge r$, $d=d'k$ and $r=r'k$.

\spa\begin{lem}\label{key}
Let $\rho \in B(G)$ be a $d$-th root of $z_P$. The value of $\rho^v\delta^u$ does not depend on the choice of $(u,v)$ such that $d'u+r'v=1$, we denote it by $q(\rho)$. It is an element of $B(W)$ such that
\begin{enumerate}[\quad (a)]
\item $q(\rho)^{d\vee r}=z_P$, where $d\vee r$ is the lcm of $d$ and $r$.
\item $q(\rho)^{d'}=\delta$
\item $q(\rho)^{r'}=\rho$
\end{enumerate}
\end{lem}
\begin{proof}
This result is analogous to Proposition \ref{334} but here, instead of considering the connection between a periodic element and a Garside element, we consider the connection between two periodic elements. This Lemma could be proved using Proposition \ref{334} in a divided category but here, it is easy to give a more direct proof:

Let $(u,v)$ be such that $d'u+r'v=1$, we have
\[(\rho^v\delta^u)^{d\vee r}=\rho^{dr'v}\delta^{d'ru}=z_P^{r'v+d'u}=z_P\]
Let now $(u',v')$ be another pair such that $d'u'+r'v'=1$. It is known that $v'=v+d'q$, $u'=u-r'q$ for some integer $q$. We obtain
\[(\rho^{v'}\delta^{u'})=\rho^{v}\rho^{d'q}\delta^{u}\delta^{-r'q}=(\rho^{v}\delta^u)(\rho^{d'}\delta^{-r'})^{q}\]
and since $(\rho^{v'}\delta^{u'})^{d\vee r}=z_P$, we have $(\rho^{d'}\delta^{-r'})^{q(d\vee r)}=1$. Since $B(W)$ is without torsion, we have $\rho^{v'}\delta^{u'}=\rho^v\delta^u$ as we claimed.

Then, since $d'$ and $r$ are also coprime, we can consider $(a,b)$ such that $d'a+rb=1$, we then get
\[q(\rho)^{d'}=(\rho^{kb}\delta^a)^{d'}=\rho^{db}\delta^{ad'}=z_P^b\delta^{1-rb}=\delta\]
Likewise, by taking $x,y$ such that $dx+r'y=1$, we get $q(\rho)^{r'}=\rho$.
\end{proof}

This Lemma is the key ingredient of this section:
\begin{enumerate}[$(A)$]
\item If $\rho$ is a $d$-th root of $z_P$ in $B(G)$, then $q(\rho)$ is a $d\vee r$-th root of $z_P$ in $B(G)\subset B(W)$. So $d\vee r$ is regular for $W$ by point $(A)$ for $W$: it divides as many degrees as codegrees of $W$. But by definition of the lcm, the (co)degrees of $G$ divided by $r$ are exactly the (co)degrees of $W$ divided by $d$ and $r$ (i.e divided by $d\vee r$). So $d$ is regular for $G$.
\item If $\rho$ and $\rho'$ are two $d$-th roots of $z_P$ in $B(G)$, then $q(\rho)$ and $q(\rho')$ are conjugate by some $g\in B(W)$. But then
\begin{enumerate}[\quad -]
\item $g\delta g^{-1}=g(q(\rho)^{d'})g^{-1}=q(\rho')^{d'}=\delta$, so $g\in B(G)=C_{B(W)}(\delta)$
\item $g\rho g^{-1}=g(q(\rho)^{r'})g^{-1}=q(\rho')^{r'}=\rho'$
\end{enumerate}
and $\rho$ and $\rho'$ are conjugate by $g\in B(G)$. 
\item Let $\rho$ be a $d$-th root of $z_P$ in $B(G)$ and $w$ its image in $G$. Again thanks to Lemma \ref{key}, we have
\[C_{B(G)}(\rho)=B(G)\cap C_{B(W)}(\rho)=C_{B(W)}(\rho)\cap C_{B(W)}(\delta)=C_{B(W)}(q(\rho))\]
The same reasoning gives $C_{B(G)}(\ttilde{w})=C_{B(W)}(q(\ttilde{w}))$ and $G'=C_W(\pi(q(\ttilde{w})))$. The morphism $B(G')\to C_{B(G)}(\ttilde{w})$ is the same morphism as $B( C_W(\pi(q(\ttilde{w})))) \to C_{B(W)}(q(\ttilde{w}))$, which is known to be an isomorphism, this concludes the proof.
\end{enumerate}

It is known that $G_{31}$ and $G_{22}$ respectively admits regular embeddings into $G_{37}$ and $G_{30}$. Since both $G_{37}$ and $G_{30}$ are well-generated groups for which the theorem is already known to hold, we get that the theorem is also true for $G_{22}$ and $G_{31}$.

\section{Degrees, codegrees and isodiscriminantality}
No exceptional badly-generated group other than $G_{31}$ and $G_{22}$ admits a regular embedding inside a well-generated group. However, the preceding sections can still be used for more than these two groups. For instance, $G_7$ has the same degrees and codegrees as the group $G(12,2,2)$, which is a member of the infinite series. By Section 4, we know that our theorem holds for $G(12,2,2)$. Here we show that this is sufficient to show that it also holds for $G_7$.

\spa\begin{prop}
Let $W,W'$ be two irreducible complex reflection groups having the same degrees and codegrees. If $(A)$ (resp. $(B)$) holds for $W$, then it also holds for $W'$.
\end{prop}
\begin{proof}
The main tool for this proof is the classification of all the pairs of irreducible groups having the same degrees and codegrees. This classification is easy by direct inspection of the degrees and codegrees of each irreducible group, and summarized by the following lemma:
\begin{lem}
The only pairs of irreducible complex reflection groups having same degrees and codegrees are 
\[\begin{array}{cccc} G_5\leftrightarrow G(6,1,2) & G_{10}\leftrightarrow G(12,1,2) & G_{18}\leftrightarrow G(30,1,2) & G_7\leftrightarrow G(12,2,2)\\ G_{11}\leftrightarrow G(24,2,2) & G_{15}\leftrightarrow G(24,4,2) & G_{19}\leftrightarrow G(60,2,2) & G_{26}\leftrightarrow G(6,1,3) \end{array}\]
We could also include $G(2,2,3)\leftrightarrow G(1,1,4),~~ G(3,3,2)\leftrightarrow G(1,1,3)$ and $G(2,1,2)\leftrightarrow G(4,4,2)$, but these pairs are moreover isomorphic as complex reflection groups.
\end{lem}
Each of these pairs are pairs of isodiscriminantal groups. This is proven in general for Shephard groups in \cite[Theorem 2.25]{orliksoldiscri}. The remaining cases are groups of rank 2, which are studied in \cite[Section 2]{bannai}: for each irreducible group of rank 2 a system of basic invariants is given, associated with the corresponding discriminant. Also, since the cardinality of the center of a complex reflection group is the gcd of its degrees, two groups sharing the same degrees both have centers of the same cardinality.

The isomorphism $\varphi:B(W)\to B(W')$ induced by isodiscriminantality sends $z_{B(W)}$ to $z_{B(W')}$. As the centers of $W$ and $W'$ have the same size, it also sends $z_{P(W)}$ to $z_{P(W')}$. So the existence of $d-$th roots for $z_{P(W)}$ implies the existence of $d$-th roots of $z_{P(W')}$. As the groups $W$ and $W'$ have the same degrees and codegrees, we get that properties $(A)$ and $(B)$ for $W$ imply properties $(A)$ and $(B)$ for $W'$.\end{proof}

In order to prove property $(C)$, we are going to restrict to groups of rank 2 so that we can use Remark \ref{r2}. The only pair of groups having identical degrees and codegrees and of rank more than $2$ is $G_{26}\leftrightarrow G(6,1,3)$. Since both of these groups are well-generated, they already are known to satisfy the theorem.

Let $d$ be a regular number giving a group of rank $1$, $d'$ its associated fundamental regular number. If $\rho$ is a $d$-th root of $z_{P'}$ and $\rho'$ is an associated fundamental root, then we already know that $C_{B(W)}(\varphi^{-1}(\rho))=C_{B(W)}(\varphi^{-1}(\rho'))=\langle \varphi^{-1}(\rho')\rangle$, and the isomorphism $\varphi$ gives the desired result.

This settles the cases of $G_7,G_{11},G_{19}$ and $G_{15}$, since $G(12,2,2),G(24,2,2),G(60,2,2)$ and $G(24,4,2)$ belong to the infinite series.

\section{The two remaining exceptional groups}
By now, the only remaining groups are $G_{12}$ and $G_{13}$. We are going to study them case by case with similar Garside theoretic tools. We recall that $B_{12}$ and $B_{13}$ are Garside groups (see for instance \cite[Example 11 and 13]{thespicantin}). The relevant data are tabulated as follows:
\[\begin{array}{r|cc} & G_{12} & G_{13} \\\hline \text{Degrees} & 6~ 8 & 8~ 12 \\ \text{Codegrees} & 0~10 & 0~16 \\ \text{Presentation of $B$} & stus=tust=ustu & cabc=bcab, ~abcab=cabca\\ \Delta & stus & (abc)^3 \\ z_B & (stu)^4=\Delta^3 & (abc)^3=\Delta\\ z_P & (stu)^8=\Delta^6 & (abc)^{12}=\Delta^4\end{array}\]
The two presentations given for $B_{12}$ and $B_{13}$ also provide the presentations for the associated Garside monoid, that we denote $M$ and $N$ respectively.

The Garside automorphisms for $M$ and $N$ are respectively given by
\[\phi:\begin{cases} s\mapsto t \\ t\mapsto u \\ u\mapsto s\end{cases} \text{~~and~~}\psi=1_{N}\]

We study the $d$-th roots of $z_P=\Delta^q$, with $q=6$ or $4$ respectively. We are going to compute a presentation of $\Cc_{p'}^{q'}$, and use Theorem \ref{divi}.

We are also going to make use of Remark \ref{r2} to prove property $(C)$. We will only have to explicitly compute a presentation of $\Cc_{p'}^{q'}$ when $p'$ is its own smallest regular number. In the other cases, we will only need to check if $\Cc_{p'}^{q'}$ is connected, so we will only compute $D_{p'}^{q'}(\Delta)$ and $D_{2p'}^{2q'}(\Delta)$.

\subsection{Case of $G_{12}$}
Recall that $M$ is the Garside monoid defined by the presentation \newline $\langle s,t,u~|~stus=tust=ustu\rangle$. It admits $B_{12}$ as its group of fractions. The full twist of $B_{12}$ is $z_P=(stu)^8=\Delta^6$, which has length $24$ in $M$. A root $\rho$ of $z_P$ then has a length that divides $24$, so the ``candidate regular numbers'' are $\{1,2,3,4,6,8,12,24\}$. 

Besides, we know that the regular numbers for $B_{12}$ are $1,2,3,4,6,8$, the fundamental regular numbers are $2,6,8$ and we have $R_1=\{1,2\}$ $R_2=\{3,6\}$, $R_3=\{4,8\}$. 

\subsubsection{Proof of property $(A)$} Here we only have to show that there are no $d$-th roots of $z_P$ for $d\in \{12,24\}$. Since $12$ divides $24$, it is sufficient to show that there are no $12$-th roots of $z_P$. Since $z_P=\Delta^6$, applying Proposition \ref{334} we compute $D_1^2(\Delta)$. For $(x,y)\in \Ss^2$, we have
\[(x,y)=(x,y)^\phi=(y,x^\phi)\Leftrightarrow x=y=x^\phi\]
Therefore we have $D_1^2(\Delta)=\{x\in \Ss~|~ x^2=\Delta\}=\varnothing$. The category $\Cc_1^2$ is empty and there are no $12$-th roots of $z_P$ in $B_{12}$. This proves point $(A)$.

In order to prove points $(B)$ and $(C)$, we need to prove the connectedness (and non-emptiness) of $\Cc_p^6$ for $p=\{2,6,8\}$, and to compute a presentation for $p=\{1,3,4\}$.
\subsubsection{Proof of connectedness for non-smallest regular numbers} 
\begin{enumerate}[$\bullet$]
\item $p=2$, we compute $\Cc_1^3=M^{\phi^3}=M$, which is connected (as a monoid).
\item $p=6$, we compute $\Cc_1^1=M^{\phi}$, which is connected (it is also a monoid).
\item $p=8$, we compute $\Cc_4^3$. For $(x,y,z,t)\in \Ss^4$, we have
\[(x,y,z,t)\in D_4^3(\Delta)\Leftrightarrow xyzt=\Delta \text{~and~} (x,y,z,t)=(t,x^\phi,y^\phi,z^\phi)\]
The second condition gives $y=x^\phi,z=x^{\phi^2},t=x$, so we deduce that
\[D_4^3(\Delta)\approx\{x\in \Ss~|~ xx^\phi x^{\phi^2}x=\Delta\}=\{s,t,u\}.\]
The same reasoning gives us
\[D_8^6(\Delta)\approx\{(\alpha,\beta)\in \Ss^2~|~ (\alpha\beta)(\alpha\beta)^{\phi}(\alpha\beta)^{\phi^2}(\alpha\beta)=\Delta\}\]
with $(\alpha,\beta):\alpha\beta \to \beta \alpha^\phi$, in particular we have the following subgraph of $\Cc_4^3$: 
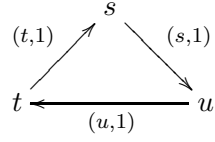
\[\xymatrix{ & s \ar@{<-}[ld]_{(t,1)} & \\
t  \ar@{<-}[rr]_{(u,1)} & &  u  \ar@{<-}[lu]_{(s,1)}}\]
which proves that $\Cc_4^3$ is connected.
\end{enumerate}

\subsubsection{Proof of connectedness and computation of the centralizer for the smallest regular numbers}
\begin{enumerate}[$\bullet$]
\item $p=1$, we compute $\Cc_1^6=M^{\phi^6}=M$, its enveloping group is $B_{12}$. We have that $z_P$ is a $1$-th root of $z_P$, and that $\Delta^3$ is an associated fundamental root. As both $z_P$ and $\Delta^3$ are central in $B_{12}$, we have $C_{B_{12}}(\Delta^3)=C_{B_{12}}(z_P)=B_{12}$ as expected.
\item $p=3$, we compute $\Cc_1^{2}=M^{\phi^2}$, which is generated by the simples $x\in \Ss$ such that $s^{\phi^2}=s$. The only such simple is $\Delta$, so we have $M^{\phi^2}=\langle \Delta \rangle$, and $\Delta$ is a $6$-th root of $z_P$.
\item $p=4$, we compute $\Cc_2^{3}$. We obtain
\begin{align*}
D_2^3(\Delta)&\approx\{x\in \Ss~|~ xx^{\phi^2}=\Delta\}=\{st,tu,us\}\\
D_4^6(\Delta)&\approx\{(\alpha,\beta)\in \Ss^2~|~ (\alpha\beta)(\alpha\beta)^{\phi^2}=\Delta\}\\
D_6^9(\Delta)&\approx\{(u,v,w)\in \Ss^3~|~ (uvw)(uvw)^{\phi^2}=\Delta\}
\end{align*}
With $(\alpha,\beta):\alpha\beta\to \beta\alpha^{\phi^2}$, and $(u,v,w)$ inducing the relation $(u,vw)(v,wu^{\phi^2})=(uv,w)$. We obtain that $\Cc_2^3$ is presented by the following graph
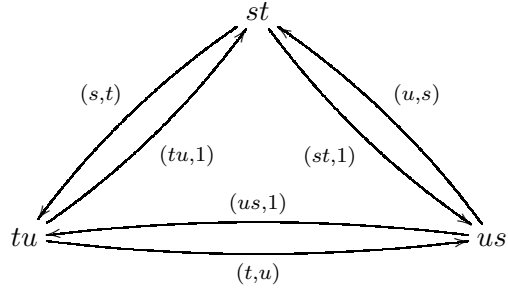
\[\xymatrixcolsep{6pc}\xymatrixrowsep{6pc}\xymatrix{ & st \ar@{<-}@/^/[ld]^{(tu,1)} \ar@{<-}@/^/[rd]^{(u,s)}& \\
tu \ar@{<-}@/^/[ur]^{(s,t)} \ar@{<-}@/^/[rr]^{(us,1)} & &  us \ar@{<-}@/^/[ll]^{(t,u)} \ar@{<-}@/^/[lu]^{(st,1)}}\]
with the relations 
\[(u,s)\circ (s,t)=(us,1); (s,t)\circ (t,u)=(st,1); (t,u)\circ (u,s)=(tu,1)\]
From this we deduce that the automorphism group of the object $st$ is cyclic and generated by $(s,t)\circ(t,u)\circ(u,s)$, which is sent by the collapse functor to $stu$, a $8$-th root of $z_P$. Since $8$ is the fundamental regular number associated to $4$, this concludes the proof of our theorem for $G_{12}$.

\end{enumerate}

\spa\subsection{Case of $G_{13}$}
Recall that $N$ is the Garside monoid defined by the presentation \newline $\langle a,b,c~|~bcab=cabc,~abcab=cabca\rangle$. It admits $B_{13}$ as its group of fractions. The full twist of $B_{13}$ is $z_P=(abc)^{12}=\Delta^4$, which has length $36$ in $N$. A root of $z_P$ then has a length that divides $36$, so the ``candidate regular numbers'' are $\{1,2,3,4,6,9,12,18,36\}$.

Besides we know that the regular numbers for $B_{13}$ are $1,2,3,4,6,12$, the fundamental regular numbers are $4,12$ , and we have $R_1=\{1,2,4\}$, $R_3=\{3,6,12\}$.

\subsubsection{Proof of property $(A)$} Here we only have to show that there are no $d$-th roots of $z_P$ for $d\in \{9,18,36\}$. Since $9$ divides $18$ and $36$, it is sufficient to show that there are no $9$-th roots of $z_P$. If $\rho$ is such a root, then $\rho^3$ is a $3$-th root of $z_P$. If we assume (for now) that point $(B)$ holds for $d=3$, then we get that $\pi(\rho)^3$ is a $3$-regular element. In particular we have $\pi(\rho)^3\neq 1$, thus $\pi(\rho)$ is of order $9$. This contradicts the fact that $9$ does not divide $|G_{13}|=96$.

In order to prove points $(B)$ and $(C)$, we need to prove the connectedness (and non-emptiness) of $\Cc_p^4$ for $p=\{2,4,6,12\}$ and to compute a presentation for $p=\{1,3\}$
\subsubsection{Proof of connectedness for non-smallest regular numbers} 
\begin{enumerate}[$\bullet$]
\item $p=2$, we compute $\Cc_1^2=N^{\psi^2}=N$, which is connected (as a monoid).
\item $p=4$, we compute $\Cc_1^1=N^{\psi}=N$, which is also connected.
\item $p=6$, we compute $\Cc_3^2$, for $(x,y,z)\in \Ss^3$. We have
\[(x,y,z)\in D_3^2(\Delta)\Leftrightarrow xyz=\Delta \text{~and~} (x,y,z)=(z,x,y)\]
The second condition gives $x=y=z$, so we deduce that
\[D_3^2(\Delta)\approx\{x\in \Ss~|~ x^3=\Delta\}=\{abc,bca,cab\}\]
The same reasoning gives us
\[D_6^4(\Delta)\approx\{(\alpha,\beta)\in \Ss~|~ (\alpha\beta)^3=\Delta\}\]
with $(\alpha,\beta):\alpha\beta\to \beta\alpha$, in particular we have the following subgraph of $\Cc_3^2$:
\[\xymatrix{cba \ar@{<-}[r]^-{(ab,c)} & abc & bca \ar@{<-}[l]_-{(a,bc)}}\]
which proves that $\Cc_3^2$ is connected.
\item $p=12$, we compute $\Cc_3^1$. The same reasoning as when $p=6$ gives $D_3^1(\Delta)=\{abc,bca,cab\}$ and $D_6^2(\Delta)=D_6^4(\Delta)$. So $\Cc_3^1$ is also connected.
\end{enumerate}

\subsubsection{Proof of connectedness and computation of the centralizer for the smallest regular numbers}
\begin{enumerate}[$\bullet$]
\item $p=1$, we compute $\Cc_1^4=N^{\psi^4}=N$. Its enveloping group is $B_{13}$. We have that $z_P$ is a $1$-root of $z_P$, and that $\Delta$ is an associated fundamental root. As both $z_P$ and $\Delta$ are central in $B_{13}$, we have $C_{B_{13}}(\Delta)=C_{B_{13}}(z_P)=B_{13}$ as expected.
\item $p=3$, we compute $\Cc_3^4$. We have
\begin{align*}
D_3^4(\Delta)&\approx\{x\in \Ss~|~x^3=\Delta\}=\{abc,bca,cab\}\\
D_6^8(\Delta)&\approx\{(\alpha,\beta)\in \Ss^2~|~ (\alpha\beta)^3=\Delta\}\\
D_9^{12}(\Delta)&\approx\{(u,v,w)\in \Ss^3~|~ (uvw)^3=\Delta\}\\
\end{align*}
With $(\alpha,\beta):\alpha\beta\to \beta\alpha$, and $(u,v,w)$ inducing the relation $(u,vw)(v,wu)=(uv,w)$. We obtain that $\Cc_2^1$ is presented by the following graph
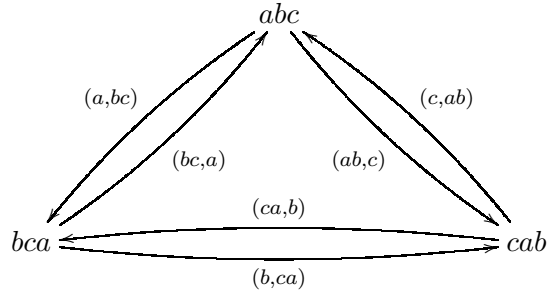
\[\xymatrixcolsep{6pc}\xymatrixrowsep{6pc}\xymatrix{ & abc \ar@{<-}@/^/[ld]^{(bc,a)} \ar@{<-}@/^/[rd]^{(c,ab)}& \\
bca \ar@{<-}@/^/[ur]^{(a,bc)} \ar@{<-}@/^/[rr]^{(ca,b)} & &  cab \ar@{<-}@/^/[ll]^{(b,ca)} \ar@{<-}@/^/[lu]^{(ab,c)}}\]
with the relations
\[\begin{array}{lll}(a,bc)(b,ca)=(ab,c) & (b,ca)(c,ab)=(bc,a) & (c,ab)(a,bc)=(ca,b)\\
(ab,c)(c,ab)=(a,bc)(bc,a) & (bc,a)(a,bc)=(b,ca)(ca,b) & (ca,b)(b,ca)=(c,ab)(ab,c) \end{array}\]
From this we deduce that the automorphism group of the object $abc$ is cyclic and generated by $(ab,c)(c,ab)$, which is sent by the collapse functor to $abc$, a $12$-th root of $\Delta$. Since $12$ is the fundamental regular number associated to $3$, this concludes the proof.
\end{enumerate}

\section{A Kerékj\'art\'o type by-product}\label{sec8}
In \cite[Remark 12.5]{beskpi1}, Bessis suggests that our main theorem can be thought of as an analogue of the Kerékj\'art\'o Theorem, stating that every periodic homeomorphism of the disk is conjugate to a rotation. In this last section we establish a rephrasing of this analogy.

Recall that the regular Springer set (cf \cite[Definition 3]{shvartsman}) of an irreducible complex reflection group $W$ is the set of regular numbers for $W$ that are maximal with respect to divisibility. We say that a $d$-th root of $z_p$ is a \nit{maximal root} of $z_p$ if $d$ is in the regular Springer set of $W$.

\begin{prop}\label{kere}
Let $W$ be an irreducible complex reflection group. If $\gamma\in B(W)$ is an element admitting a central power, then we have $\gamma=\rho^p$ for $\rho$ a maximal root of $z_P$ and some integer $p$. In particular $\pi(\gamma)$ is a regular element in $W$.
\end{prop}

If $B(W)$ is a Garside group with a $\Delta$ such that $z_B$ is a power of $\Delta$, then an element of $B(W)$ admits a central power in $B(W)$ if and only if it is periodic in the Garside sense. So the two notions of periodic elements are the same. This also justifies the connection between Proposition \ref{kere} and the Kerékj\'art\'o Theorem.

\begin{proof} This statement is already known to hold for the infinite series by Lemma 4.1 and proposition 4.2. Indeed $\varepsilon$ (resp. $\lambda$) is the only fundamental root of $z_P$ in $B(de,1,n)$ (resp. $B(de,e,n)$) up to conjugacy. Let us first prove proposition \ref{kere} in the case where $B(W)$ is a Garside group, and $z_B=\Delta^n$. If $\gamma^p=\Delta^q$ in $B(W)$, then we have $\gamma^{pn}=\Delta^{qn}=z_B^q$.

If we note $k=(pn)\wedge q$, we have that $\gamma$ is conjugate to some $\ttilde{\gamma}$ such that $\ttilde{\gamma}^{\frac{pn}{k}}=z_B^{\frac{q}{k}}$. But we know that $z_B=\Delta^n$ is also a Garside element, which is central. Using proposition \ref{334}, there is some $\rho$ such that $\ttilde{\gamma}=\rho^{\frac{pn}{k}}$, where we have $\rho^{\frac{q}{k}}=z_B$. This settles the case of almost every irreducible group, the exception being $G_{31}$. Let $\delta$ be a $4$-th root of $z_P$ in $B_{37}$. We know that $B_{31}$ is isomorphic to $C_{B_{37}}(\delta)$. We have $Z(C_{B_{37}}(\delta))=\langle \delta\rangle$ (see \cite[Corollary 12.7]{beskpi1}), that is the isomorphism $C_{B_{37}}(\delta)\simeq B_{31}$ sends $\delta$ to $z_B$. The case of $B_{31}$ is then a consequence of Lemma \ref{key}.\end{proof}

We thank Referee \#1 for bringing to our attention the article \cite{shvartsman}, which studies the quotient group of the braid group of a complexified Coxeter group by its center. The description of periodic elements we provide in the case of an irreducible complex braid group allows us to give a generalization of \cite[Theorem B]{shvartsman}:

Let $W$ be an irreducible complex reflection group. As all fundamental regular numbers are divisible by the cardinality $|Z(W)|$ of the center of $W$, we can define the \nit{modified Springer set} as the set of elements of the regular Springer set divided by $|Z(W)|$.

\begin{prop}(Shvartsman)\newline Let $W$ be an irreducible complex reflection group. Let $\gamma$ be an element of finite order in the group $B(W)/Z(B(W))$. Then the order of $\gamma$ divides one of the numbers belonging to the modified Springer set of $W$. Conversely, for each divisor $t$ of an element of the modified Springer set, there is an element $\gamma\in B(W)/Z(B(W))$ having order $t$.
\end{prop}
\begin{proof}
Let $\gamma$ be an element of finite order. There is some periodic element $x\in B(W)$ having image $\gamma$ in the quotient $B(W)/Z(B(W))$. By Proposition \ref{kere}, $x$ is equal to some $\rho^p$ where $\rho$ is a maximal root of $z_P$ and $p$ is a positive integer. Since $\rho$ is a maximal root of $z_P$, there is an element $d$ of the regular Springer set such that $\rho^d=z_P$. The integer $d'=\frac{d}{|Z(W)|}$ is an element of the modified Springer set such that $\rho^{d'}=z_B$. Let $r$ be the image of $\rho$ in $B(W)/Z(B(W))$. We have $\gamma=r^{p}$, so the order of $\gamma$ is $\frac{p\vee d'}{p}=\frac{d'}{p\wedge d'}$, where $p\vee d'$ is the lcm of $p$ and $d'$, and $p\wedge d'$ is their gcd. So the order of $\gamma$ divides $d'$.

Conversely, let $d$ be an element of the regular Springer set, let $d'=\frac{d}{|Z(W)|}$ be the associated element of the modified Springer set. Let $\rho$ be a $d$-th root of $z_P$, it induces an element of order $d'$ in $B(W)/Z(B(W))$, so considering powers of $\rho$ gives the desired results.
\end{proof}

\newpage
\printbibliography
\end{document}